\documentclass{amsart}
        \usepackage{amsmath}
        \usepackage{amsfonts}
        \usepackage{amsthm}
        \usepackage[hypertexnames=false]{hyperref}
        \usepackage{comment}
        \usepackage{url}
        \usepackage{color}
        \usepackage{enumitem}
        \setlist[enumerate]{font=\normalfont}
        \usepackage{mathtools}
		\usepackage{amsrefs}
        \usepackage{latexsym}
        \usepackage{amssymb}
		\usepackage{graphicx}        
        
        \usepackage{eucal}

        \theoremstyle{plain}
        \newtheorem*{theorem*}{Theorem}
        \newtheorem{theorem}{Theorem}[section]
        \newtheorem{corollary}[theorem]{Corollary}
        \newtheorem{lemma}[theorem]{Lemma}

        \theoremstyle{definition}
        \newtheorem{definition}[theorem]{Definition}

        \newtheorem*{example*}{Example}

        \theoremstyle{remark}
        \newtheorem{remark}[theorem]{Remark}
        \newtheorem*{remark*}{Remark}

\newcommand{\PP}{\mathbb{P}}
\newcommand{\RR}{\mathbb{R}}

\newcommand{\EE}{\mathbb{E}}

\begin{document}

\author{Hern\'an Garc\'ia}
\address{
Departamento de matem\'aticas\\
Universidad de los Andes\\
Carrera $1^{\rm ra}\#18A-12$\\ 
Bogot\'a, Colombia
}
\email{jh.garcia1776@uniandes.edu.co, mj.junca20@uniandes.edu.co, mvelasco@uniandes.edu.co}

\author{Camilo Hern\'andez}
\address{
Industrial Engineering and Operations Research Department\\ 
Columbia University\\
$500$ West $120$-th street\\ 
New York, NY 10027\\
}
\email{camilo.hernandez@columbia.edu}

\author{Mauricio Junca}
\address{
Departamento de matem\'aticas\\
Universidad de los Andes\\
Carrera $1^{\rm ra}\#18A-12$\\ 
Bogot\'a, Colombia
}
\email{mj.junca20@uniandes.edu.co}

\author{Mauricio Velasco}
\address{
Departamento de matem\'aticas\\
Universidad de los Andes\\
Carrera $1^{\rm ra}\#18A-12$\\ 
Bogot\'a, Colombia
}
\email{mvelasco@uniandes.edu.co}


\begin{abstract} We propose convex optimization algorithms to recover a good approximation of a point measure $\mu$ on the unit sphere $S\subseteq \RR^n$ from its moments with respect to a set of real-valued functions $f_1,\dots, f_m$. Given a finite subset $C\subseteq S$ the algorithm produces a measure $\mu^*$ supported on $C$ and we prove that $\mu^*$ is a good approximation to $\mu$ whenever the functions $f_1,\dots, f_m$ are a sufficiently large random sample of independent Kostlan-Shub-Smale polynomials.
More specifically, we give sufficient conditions for the validity of the equality $\mu=\mu^*$ when $\mu$ is supported on $C$ and prove that $\mu^*$ is close to the best approximation to $\mu$ supported on $C$ provided that all points in the support of $\mu$ are close to $C$.
 
\end{abstract}

\title{Compressive sensing and truncated moment problems on spheres.}
\maketitle

\section{Introduction}

Let $K\subseteq \RR^n$ be a compact set and let $V$ be a vector space of continuous real-valued functions on $K$. The truncated moment problem defined by $V$ consists of the following two parts:

\begin{enumerate}
\item Characterizing the convex cone $\mathcal{M}(V)$ of linear operators $L:V\rightarrow \RR$ which are {\it representable in $V$ by measures} i.e. those for which there exists a finite Borel measure $\mu$ on $K$ satisfying $L(f)=\int_K fd\mu$ for all $f\in V$.
\item Finding a reconstruction procedure which, given a representable operator $L\in \mathcal{M}(V)$, produces a finite Borel measure $\mu$ which represents $L$.
\end{enumerate} 

Truncated moment problems play a central role in modern convex optimization because they allow us to convexify, and often solve, optimization problems of the form $\max_{x\in K} f(x)$. It is easy to see that if $f\in V$ then the optimal value of this problem is equal to the optimal value of the convex optimization problem $\max_{L\in \mathcal{M}(K),L(1)=1} L(f)$ and that any measure $\mu$ which represents a maximizer $L^*$ is supported at points of $K$ where $f$ achieves its maximum value. This approach is one of the main methods for solving polynomial optimization problems in practice (see for instance the book~\cite{L} and its extensive reference list).

Discrete measures (i.e. conic combinations $\sum_{i=1}^k c_i\delta_{x_i}$ of Dirac delta measures supported at points $x_i\in K$) play a fundamental role in the solution of both $(1)$ and $(2)$ above. The main reason is that under rather general circumstances the cone $\mathcal{M}(V)\subseteq V^*$ coincides with the cone of discrete measures (see Lemma~\ref{Lem: basicTM} for a precise statement).

In this article we therefore focus on problem $(2)$ above for discrete measures. Our main result is to propose new approximate reconstruction procedures when $K=S^{n-1}$ is the unit sphere in $\RR^n$ and $V$ is the set of polynomials of degree $\leq d$ for some integer $d>0$. More concretely, we ask: {\it Given the vector $b_{\mu}$ of monomial {\it moments of degree $\leq d$} of a discrete measure $\mu:=\sum_{i=1}^k g_i\delta_{x_i}$ defined by $(b_{\mu})_{\alpha}=\int_K x^{\alpha}d\mu$, how to find a discrete measure $\mu^*$ which is a good approximation of $\mu$?} 

Our proposal is to choose a sufficiently dense code $C=\{q_1,\dots, q_N\}\subseteq S^{n-1}$ and try to find a measure $\mu^*$ {\it supported on $C$} which is a good approximation to $\mu$. In this setting the points of $C$ and the $m$ monomial functions of degree $\leq d$ determine a linear measurement map $M:\RR^N\rightarrow \RR^m$ which sends the coefficients $(c_1,\dots, c_N)$ of a discrete measure $\sum c_i\delta_{q_i}$ supported on $C$ to its vector of monomial moments $(b)_{\alpha}:=\sum_{i=1}^Nc_ix^{\alpha}(q_i)$. Using $M$ we can reintepret the problem of finding a good approximation for $\mu$ into one of finding {\it approximate solutions to a system of linear equations}. This formulation allows us to think of the problem as an instance of {\it compressed sensing} and to address it via convex optimization. In particular, we can apply the remarkable results of Cand\'es, Donoho, Romberg, Tao and others (see for instance \cite{CandesRombergTao}, \cite{CandesTao},\cite{Donoho}) and obtain recovery guarantees. Our first result, proven in in Section~\ref{Sec: BasicCS}, gives such recovery guarantees in terms of the restricted isometry constants $\delta_{2k}(M)$ of the map $M$ (see Section~\ref{Prelim: RIP} for precise definitions):

\begin{theorem*}[{\bf A}] \label{CompressedComputing} 
Assume the inequality $\delta_{2k}(M)<\sqrt{2}-1$ holds. Then there exists a constant $B_1$ such that
\begin{enumerate}
\item{{\bf Exact recovery:} If ${\rm supp}(\mu)\subseteq C$ and $c^*$ is a minimizer of the problem 
\[\min \|c\|_1\text{ subject to $Mc=b_\mu$.}\] then $\mu=\sum_{i=1}^N c^*_i\delta_{q_i}$.}
\item{{\bf Approximate recovery:} Assume there exists a measure $\nu=\sum_{i=1}^k (c_\nu)_i\delta_{q_{j(i)}}$ supported on $C$ with $\|b_{\mu}-b_{\nu}\|_2\leq \tau$. If $c^*$ is a minimizer of the problem
\[\min \|c\|_1\text{ subject to $\|Mc-b_\mu\|_2\leq 
\tau$.}\]
then $\|c^*-c_\nu\|_2\leq B_1\tau$.
}
\end{enumerate}
\end{theorem*}

We do not know the value of the restricted isometry constants $\delta_{2k}(M)$ for the map $M$ and due to the well known ill-conditioning of the Vandermonde matrices we do not expect them to be small in general (i.e. for all point configurations $C$). 

While it may be difficult to determine the isometry constants for a given linear map (the problem is known~\cite{CCRip} to be NP-hard) the literature in compressed sensing has emphasized since its inception that it is much easier to understand the restriced isometry constants of {\it random} matrices (for instance of those with independent standard normal entries). In that spirit we ask whether it is possible to {\it randomize} the measurement matrix $M$ replacing the monomial basis by a basis {\it consisting of random polynomials} in order to obtain more explicit recovery guarantees. 

In Section~\ref{Sec:Random} we show that the answer to this question is affirmative. The key idea is that there is a natural probability measure on the space of polynomials of degree at most $d$ on $S^{n-1}$ which is invariant under the natural action of the orthogonal group $O(n)$. This is the well-known Kostlan-Shub-Smale measure~\cite{Kostlan},~\cite{ShubSmale}, which is explicitly given by polynomials  $P(x)=\sum_{\alpha: |\alpha|\leq d} A_{\alpha} x^{\alpha}$ where the coefficients $A_{\alpha}$ are independent Gaussian random variables with mean zero and variance $\frac{\binom{d}{\alpha}}{2^d}$. The orthogonal invariance allows us to relate the {\it average} behavior of the restricted isometry constants of the measurement matrices defined by an independent sample of Kostlan-Shub-Smale polynomials $f_1,\dots, f_m$ of degree $d$ with the geometry of the points in $C$. In Section~\ref{Sec:Random} we combine this idea with the appropriate concentration inequalities and prove the following Theorem,

\begin{theorem*}[{\bf B}]\label{main} Let  $\Phi_{ij}:=\frac{f_i(q_j)}{\sqrt{m}}$. For any real number $0<\delta<\frac{1}{2} $, any integer $k$, and all sufficiently large $d$ the following inequality holds
\[ \PP\left\{\delta_{2k}(\Phi)>\delta\right\}< \binom{N}{2k}\left(\frac{30}{\delta}\right)^{2k}2e^{-mc_0\left(\frac{\delta}{6}\right)}\] 
where $c_0(\eta):=\min\left(\frac{1}{2}\log\left(\frac{1}{1+\eta}\right)+\frac{\eta}{2}, \frac{1}{2}\log\left(\frac{1}{1-\eta}\right)-\frac{\eta}{2}\right).$

In particular, for all sufficiently large $d$, there exist matrices $\Phi$ such that $\delta_{2k}(\Phi)<\delta$ whenever 
\[m\geq \frac{2k\log\left(\frac{30eN}{2k\delta}\right)+\log(2)}{c_0\left(\frac{\delta}{6}\right)}\]
\end{theorem*}

Combining the previous two theorems we are able to prove our main result, which guarantees that the proposed method recovers good approximations of measures from moments with respect to a sufficiently large sample of Kostlan-Shub-Smale polynomials, provided the measure is close to the points of our grid $C$. More precisely, if $\mu:=\sum_{i=1}^kg_i\delta_{x_i}$ is a point measure with given moments $(b_{\mu})_i:=\int_K f_id\mu$ and $\theta:=\max_j\min_i\arccos\langle x_j,q_i\rangle$ then the following Theorem holds,

\begin{theorem*}[{\bf C}]\label{main:Approx} For any $\epsilon>0$ the optimization problem 

\[\min \|c\|_1\text{ s.t. $\|Mc-b_{\mu}\|\leq (1+\epsilon)  \|g\|_2\sqrt{2k\left(1-\left(\frac{1+\cos(\theta)}{2}\right)^d\right)}$} \]

has as optimal solution a vector $z^*$ satisfying
\[ \|g-z^*\|_2\leq B_1(1+\epsilon)\|g\|_2\sqrt{2k\left(1-\left(\frac{1+\cos(\theta)}{2}\right)^d\right)}\]
for all sufficiently large $d$ with probability at least
\[ 1- \binom{N}{2k}\left(\frac{30}{\sqrt{2}-1}\right)^{2k}2e^{-mc_0\left(\frac{\sqrt{2}-1}{6}\right)} - 2e^{-mc_0(\epsilon)}\]
where $c_0(\eta):=\min\left(\frac{1}{2}\log\left(\frac{1}{1+\eta}\right)+\frac{\eta}{2}, \frac{1}{2}\log\left(\frac{1}{1-\eta}\right)-\frac{\eta}{2}\right).$
\end{theorem*}

Motivated by these results we propose the following second-order cone programming procedure for approximate recovery of a discrete measure $\mu$ from its vector of moments $b_{\mu}$ with respect to any set of functions $h_1,\dots, h_m$.

\begin{enumerate}
\item Fix a code $C=\{q_1,\dots, q_N\}\subseteq S^{n-1}$ and define $M_{ij}=h_i(q_j)$.
\item Fix a small tolerance parameter $\tau$ and let $c^*$ be the solution of the convex optimization problem
\[\min \|c\|_1\text{ s.t. $\|Mc-b_{\mu}\|_2\leq \tau$}\]
\item Return $\mu^*:=\sum_{i=1}^N\max(c_i^*,0)\delta_{q_i}$.
\end{enumerate}
Our numerical experiments (see Section~\ref{Sec: numericAlgos}) suggest that this procedure works well for sufficiently small $\tau$ in practice and that it can be solved efficiently even for large instances. 

Finally we prove in Theorem~\ref{thm: consistency} that if $\mu$ is a probability measure then our approximation algorithm can be applied to a sequence $C_{(j)}$ of successively denser codes to obtain a consistent estimation of $\mu$, that is, to construct a sequence of probability measures $\mu_{C_{(j)}}^*$ supported on $C_{(j)}$ which converges to $\mu$ in the Wasserstein metric, assuming we know the moments of $b_{\mu}$ with respect to sufficiently many Kostlan-Shub-Smale polynomials.

\smallskip

{\bf Acknowledgements.} We wish to thank Diego Armentano, Greg Blekherman and Fabrice Gamboa for useful conversations during the completion of this project.
M. Junca and M. Velasco were partially supported by the FAPA funds from Universidad de los Andes.

\section{Preliminaries}

{\bf Notation.} For an integer $N$ let $[N]=\{1,2,\dots, N\}$. For an $m\times N$ matrix $A$ and a subset $S\subseteq [N]$ we let $A_S$ be the $m\times |S|$ submatrix of $A$ consisting of the columns of $A$ indexed by the elements of $S$. A multi-index $\alpha$ with $n$-parts is an $n$-tuple $\alpha=(\alpha_1,\dots \alpha_n)$ of natural numbers $\alpha_i$. A multi-index $\alpha$ with $n$-parts has a degree $|\alpha|:=\alpha_1+\dots +\alpha_n$ and an associated monomial $x^{\alpha}:=\prod_{i=1}^n x_i^{\alpha_i}$ in the ring of polynomials $\RR[x_1,\dots, x_n]$. A vector $x\in \mathbb{R}^N$ is called $s$-sparse if it has at most $s$ non-zero components. 

\subsection{Truncated moment problems}
As in the introduction let $K\subseteq \RR^n$ be a compact set and let $V$ be a finite-dimensional vector subspace of the space $C(K,\RR)$ continuous real-valued functions on $K$. By a discrete measure on $K$ we mean a conic combination of Dirac delta measures supported at points of $K$. If $\nu$ is a finite Borel measure on $K$ let $L_{\nu}: C(X,\RR)\rightarrow \RR$  be the map given by $L_{\nu}(f):=\int_K f d\nu$. We say that an operator $L:V\rightarrow \RR$ is representable by a measure if there exists a finite Borel measure $\nu$ such that $L(f)=L_{\nu}(f)$ for every $f\in V$. The following Lemma, which we learned from Greg Blekherman~\cite{GB}, explains the key role played by discrete measures in truncated moment problems. It is a generalization of results of Tchakaloff~\cite{Tchakaloff} and Putinar~\cite{Putinar}.

\begin{lemma} \label{Lem: basicTM} If the functions in $V$ have no common zeroes on $K$ then every linear operator $L\in V^*$ representable by a measure is representable by a discrete measure with at most $\dim(V^*)+1$ atoms.
\end{lemma}
\begin{proof} Let $P\subseteq V$ be the closed convex cone of functions in $V$ which are nonnegative at all points of $K$. It is immediate that $P={\rm Conv}(L_{\delta_x}:x\in K)^*$. By the bi-duality Theorem from convex geometry we conclude that $P^*= \overline{{\rm Conv}(L_{\delta_{x}}:x\in K)}$. 
Now consider the map $\phi: K\rightarrow V^*$ sending a point $x$ to the restriction of $L_{\delta_x}$ (i.e. to the evaluation at $x$). This map is continuous and therefore $S:=\phi(K)$ is a compact set. Since the functions in $V$ have no points in common the convex hull of $S$ does not contain zero and therefore the cone of discrete measures ${\rm Conv}(L_{\delta_x}:x\in K)$ is closed in $V^*$. Let $\mathcal{M}(V)\subseteq V^ *$ be the cone of operators representable by a finite borel measure. Since 
${\rm Conv}(L_{\delta_x}:x\in K) \subseteq \mathcal{M}(V)\subseteq P^*$ we conclude that $\mathcal{M}(V)$ equals the cone of discrete measures as claimed. The bound on the number of atoms follows from Caratheodory's Theorem~\cite{Barvinok}.
\end{proof}

\subsection{The restricted isometry property}
\label{Prelim: RIP}

In this section we recall some basic facts about the restricted isometry property, introduced by Cand\'es and Tao in~\cite{CandesTao}.

\begin{definition} Let $A: \mathbb{R}^N\rightarrow \mathbb{R}^m$ be a linear map. The $s$-th isometry constant of $A$ is the smallest real number $\delta_s$ such that the following inequalities hold
\[(1-\delta_s)\|x\|_2^2\leq \|Ax\|_2^2\leq (1+\delta_s)\|x\|_2^2 \] 
for all $s$-sparse vectors $x\in \mathbb{R}^N$.
Equivalently, $\delta_s$ is the smallest real number such that the eigenvalues of the positive semidefinite matrices $A_S^tA_S$ are contained in the interval $[1-\delta_{s},1+\delta_{s}]$ for all $S\subseteq [N]$ with $|S|=s$.
\end{definition}

The importance of the restricted isometry property in the context of compressive sensing is summarized in Theorem~\ref{RIPrelevance} below, due to Cand\'es, Romberg and Tao~\cite{CandesRombergTao}. For a vector $x\in \mathbb{R}^n$ and a positive integer $s$ we let $x_s$ be the best $s$-sparse approximation to $x$ defined as the vector with all but the $s$-entries with largest absolute value of $x$ set to zero.

\begin{theorem}\label{RIPrelevance} \cite[Theorems 1.1 and 1.2]{Candes}. Let $k$ be any positive integer. If $A\in \RR^{m\times N}$ satisfies $\delta_{2k}(A)<\sqrt{2}-1$ then there exist constants $B_0,B_1$ such that the following statements hold for any $x\in\mathbb{R}^N$
\begin{enumerate}
\item The solution 
$x^*$ to the problem
\[\min\|z\|_1\text{ subject to $Az=Ax$}\] 
satisfies the inequality 
$\|x-x^*\|_2\leq B_0\|x-x_k\|_1$. 
\item If $y=Ax+w$ where $w$ is an unknown noise term with $\|w\|_2\leq \eta$ then the solution $x^*$ to the problem
\[\min\|z\|_1\text{ subject to $\|Az-y\|_2\leq \eta$}\] 
satisfies the inequality
\[\|x-x^*\|_2\leq B_0k^{-\frac{1}{2}}\|x-x_k\|_1+B_1\eta.\]
\end{enumerate}
\end{theorem}
\begin{remark} Part $(1)$ implies that the recovery is exact if $x$ is $k$-sparse. Part $(2)$ implies that the magnitude of the recovery error is essentially bounded by the size $\eta$ of the measurement error when $x$ is $k$-sparse.
\end{remark}
\begin{remark} The explicit values of the rather small constants $B_0$ and $B_1$ above appear in the proofs of~\cite{Candes} (see also Remark after Theorem 1.2 in ~\cite{Candes}).
\end{remark}

\section{Compressive sensing of point measures on spheres}
\label{Sec: BasicCS}
By a point measure in the unit sphere $S:=S^{n-1}\subseteq \RR^n$ we mean a linear combination $\nu:=\sum c_i\delta_{q_i}$ of Dirac delta measures centered at $N$ distinct points $q_1,\dots, q_N\in S$ with coefficients $c_i\in \RR_+$. We say that the measure $\nu$ is supported on the set $T\subseteq \{q_1,\dots, q_N\}$ if $c_i=0$ for all $i$ with $q_i\not\in T$ and denote by $c_{\nu}\in \RR^N$ the vector of coefficients. 

\begin{definition} If $f_1,\dots, f_m$ are a sequence of real-valued functions on the sphere then the vector of moments of $\nu$ with respect to the $f_i$'s is the vector $b_\nu\in\mathbb{R}^m$ with components $(b_{\nu})_j =\int_Sf_jd\nu=\sum c_i f_j(q_i)$. By the vector of moments of degree $d$ of $\nu$ we mean the special case when the functions are the monomials of total degree at most $d$ in $n$ variables in the lexicographic order.\end{definition} 

In this article we will fix a set of points $C=\{q_1,\dots, q_N\}\subseteq S$, a sequence of functions $f_1,\dots, f_m$ and let $\mu:=\sum_{i=1}^kc_i \delta_{y_i}$ for some $y_i\in S$ and $c_i\in \RR_+$ and study the following two problems:

\begin{enumerate}
\item {\it Exact Recovery.} Recover the measure $\mu$ from its vector of moments $b_{\mu}$, knowing that the points $\{y_i\}_{i=1}^k$ which support $\mu$ are an (unknown) subset of the grid points $\{q_i\}_{i=1}^N$.
\item {\it Approximate Recovery.} Given the vector of moments $b_{\mu}$ of $\mu$, find a measure $\mu^*$, supported on the points $q_i$ which is close to the best approximation $\nu$ to $\mu$ supported on a set of at most $k$ of the grid points $q_i$. By approximate recovery we mean finding a vector $c^*\in \RR^N$ which differs from $c_\nu$ by a small error in the $\ell_2$ norm.
\end{enumerate}

Our fist Theorem relates the two problems above with the compressed sensing framework. In particular it shows that under certain assumptions, both problems can be addressed via convex optimization. 

\begin{theorem*}[{\bf A}] Let $M$ be the $m\times N$ matrix given by $M_{ij}=f_i(q_j)$. Assume the inequality $\delta_{2k}(M)<\sqrt{2}-1$ holds. Then there exists a constant $B_1$ such that
\begin{enumerate}
\item{{\bf Exact recovery:} If ${\rm supp}(\mu)\subseteq C$ and $c^*$ is a minimizer of the problem 
\[\min \|c\|_1\text{ subject to $Mc=b_\mu$.}\] then $\mu=\sum_{i=1}^N c^*_i\delta_{q_i}$.}
\item{{\bf Approximate recovery:} Assume there exists a measure $\nu=\sum_{i=1}^k (c_\nu)_i\delta_{q_{j(i)}}$ supported on $C$ with $\|b_{\mu}-b_{\nu}\|_2\leq \tau$. If $c^*$ is a minimizer of the problem
\[\min \|c\|_1\text{ subject to $\|Mc-b_\mu\|_2\leq 
\tau$.}\]
then $\|c^*-c_\nu\|_2\leq B_1\tau$.
}
\end{enumerate}
\end{theorem*}
\begin{proof} If $c\in \mathbb{R}^N$ then the vector $Mc$ equals the vector of moments with respect to the functions $f_j$ of the measure $\sum_{i=1}^N c_i\delta_{q_i}$. The first claim is thus  a direct consequence of Theorem~\ref{RIPrelevance} part $(1)$. For the second claim let $c_\nu$ be the vector of coefficients of any $k$-sparse measure $\nu$. Letting $w=b_\mu-b_\nu$ we see that the equality $b_\mu = Mc_\nu +w$ holds and that $\|w\|_2\leq \tau$. It follows from Theorem~\ref{RIPrelevance} part $(2)$ that a minimizer $c^*$ of the problem in part $(2)$ above satisfies $\|c^*-c_\nu\|_2\leq B_1\tau$ because $c_\nu$ is $k$-sparse.
\end{proof}

The quality of the previous algorithm depends on how adequate for compressive sensing is the measurement matrix $M$. The answer will depend on the points $q_i$ and on the functions $f_j$ and is, in general, a difficult problem, in the sense that computing the restricted isometry constants $\delta_{2k}(M)$ of a matrix $M$ is an NP-hard problem~\cite{CCRip}.  In the next section we will prove that it is possible to find good functions for any point set by sampling random polynomials with a carefully chosen measure.

\section{Random polynomials for compressive sensing of point measures.}
\label{Sec:Random}
We begin by defining a probability measure on the space of polynomials of degree at most $d$. It is shown in ~\cite[Part II]{Kostlan} that there is an orthogonally invariant probability measure on homogeneous polynomials of degree $d$ for which the coefficients are independent and moreover that this measure is unique up to a common scaling of the coefficients. Our measure is obtained by a weighted combination of such invariant measures.

\begin{definition} For a multi-index $\alpha$ with $|\alpha|\leq d$ define 
\[\binom{d}{\alpha}:=\frac{d!}{(d-|\alpha|)!\prod(\alpha_i!)}\]
and let $A_{\alpha}$ be a normal random variable with mean zero and variance $\frac{\binom{d}{\alpha}}{2^d}$. Assume moreover that the random variables $A_{\alpha}$ are independent for distinct multi-indices $\alpha$. Let $P$ be the random polynomial of degree at most $d$ given by 
\[ P(x):=\sum_{\alpha: |\alpha|\leq d} A_{\alpha}x^{\alpha}.\]
We will refer to these random polynomials as Kostlan-Shub-Smale polynomials of degree $d$.
\end{definition}
\begin{definition}
Given the set of points $C=\{q_1,\dots, q_N\}\subseteq S^{n-1}$ and a positive integer $m$ let $P_1(x),\dots, P_m(x)$ be an independent sample of $m$ Kostlan-Shub-Smale polynomials and let $X$ be the $m\times N$ matrix given by $X_{ij}:=P_i(q_j)$. Define the normalized measurement matrices $\Phi:=\frac{1}{\sqrt{m}}X$.
\end{definition}

\begin{remark} Note that the matrix $X$ depends on the integers $d$ and $m$ and on the chosen set of points $C=\{q_1,\dots, q_N\}$, however to ease the notation we will write $X$ in place of $X(d,m,C)$.
\end{remark}

The next Lemma summarizes the main statistical properties of the random matrices $\Phi$.

\begin{lemma}\label{Basic} The following statements hold: 
\begin{enumerate}
\item The vector $(P(q_j))_{1\leq j\leq N}$ is normally distributed and has mean zero. Its variance-covariance matrix is the matrix $V$ with $V_{st}=\left(\frac{1+\langle q_s, q_t\rangle}{2}\right) ^d$. In particular, for any vector $c\in \mathbb{R}^N$ we have
\[\mathbb{E}\left[\left\|\Phi c\right\|_2^2\right]=c^tVc\]

\item For any set $S\subseteq [N]$, the matrix $X_S^tX_S$ has the Wishart distribution $W(V_{S},m)$ where $V_S$ is the matrix obtained from $V$ by restriction to the rows and columns indexed by the elements of $S$.
\end{enumerate}
\end{lemma}
\begin{proof}$(1)$ The random variable $P(q_j)$ is a linear combination of normal random variables with mean zero. It is therefore normal and has mean zero. Its variance-covariance matrix is given by
\[V_{st}=\mathbb{E}\left[P(q_s)P(q_t)\right]=\mathbb{E}\left[\left(\sum_{\alpha: |\alpha|\leq d} A_{\alpha}x^{\alpha}(q_s)\right)\left(\sum_{\beta: |\beta|\leq d} A_{\beta}x^{\beta}(q_t)\right)\right]=\]
\[=\sum_{\alpha,\beta} \mathbb{E}\left[A_{\alpha}A_{\beta}\right]x^{\alpha}(q_s)x^{\beta}(q_t)=\frac{1}{2^d}\sum_{\alpha: |\alpha|\leq d} \binom{d}{\alpha}x^{\alpha}(q_s)x^{\alpha}(q_t)\]
where the equality follows from the fact that the random variables $A_\alpha$ and $A_\beta$ have mean zero and are independent for $\alpha\neq \beta$. Since $\binom{d}{\alpha}=\binom{d}{|\alpha|}\binom{|\alpha|}{\alpha}$, the last quantity equals 
\[\frac{1}{2^d}\sum_{k=0}^d \binom{d}{k}\left(\sum_{\alpha: |\alpha|=k} \binom{d}{\alpha}x^{\alpha}(q_s)x^{\alpha}(q_t)\right)=\frac{1}{2^d}\sum_{k=0}^d\binom{d}{k}\left(\sum_{r=1}^nx_r(q_s)x_r(q_t)\right)^d=\]
\[=\frac{1}{2^d}\sum_{k=0}^d\binom{d}{k}\langle q_s,q_t\rangle^d = \left(\frac{1+\langle q_s,q_t\rangle }{2}\right)^d\]
proving the claim. $(2)$ The $m$ rows of the matrix $X_S$ are independently drawn from an $|S|$-variate normal distribution with zero mean and variance-covariance matrix obtained from $V$ by restricting to the rows and columns indexed by elements of $S$. The distribution of $X_S^tX_S$ is thus, by definition, the Wishart distribution $W(V_{S},m)$.

\end{proof}

\begin{remark} The previous Lemma shows that our measure is ``normalized" so that $\EE[P(x)^2]=1$ for every point $x$ in the unit sphere. This explains our choice of $2^d$ in the denominator. \end{remark}

By the previous Lemma, the expected value of the matrix $\Phi_S^t\Phi_S$ is precisely $V_S$. The following Lemma shows that, if the $q_i$ are not too close together, in the sense that the the cosine of the angle between every two distinct vectors is bounded above by a number $\alpha<1$ then the eigenvalues of the matrices $V_S$ concentrate around one very quickly as $d$ increases. As a result, the average of the matrices $\Phi_S^t\Phi_S$ has all its eigenvalues close to one. In the following section we will use concentration inequalities to show that this implies a similar behavior for $\Phi_S^t\Phi_S$ for all subsets $S$ of a given size with high probability.

\begin{lemma}\label{Gershgorin} If $S\subseteq [N]$ has cardinality $k$ and for $i,j\in S$ with $i\neq j$ we have $\langle q_i\cdot q_j\rangle \leq \alpha < 1$ then the following eigenvalue inequalities hold:
\begin{enumerate}
\item $\lambda_{\max}(V_S)\leq 1+(k-1)\left(\frac{1+\alpha}{2}\right) ^d$
\item $\lambda_{\min}(V_S)\geq 1-(k-1)\left(\frac{1+\alpha}{2}\right) ^d$
\end{enumerate}
In particular, the eigenvalues of $V_S$ concentrate around one as $d\rightarrow \infty$.
\end{lemma}
\begin{proof} Since the points $q_i$ lie in the unit sphere, the diagonal entries of the matrix $V_S$ equal one. By our assumption on the $q_i$, the off-diagonal entries of $V$ have absolute value at most $\left(\frac{1+\alpha}{2}\right) ^d$. By the Gershgorin circle Theorem we conclude that the eigenvalues of $V_S$ are contained in the circle centered at one and with radius $(k-1)\left(\frac{1+\alpha}{2}\right) ^d$ proving the claim.
\end{proof}

\subsection{A probabilistic algorithm for compressive sensing of point measures.}

Let $k$ be any integer and let $\delta$ be a real number in $(0,1)$. In this section we estimate the probability of the set of $m\times N$ matrices $\Phi$ for which $\delta_{2k}(\Phi)>\delta$ as a function of $d$ and $m$. Our main result is Theorem~\hyperref[main]{B} showing that this probability decreases quickly as $m$ and $d$ grow. These estimates will lead to Corollary~\ref{ExactRecovery} which gives a probabilistic algorithm for compressive sensing of point measures. Our proof adapts the proof proposed by Baraniuk, Davenport, DeVore and Wakin in~\cite{BDDW} of the classical results on compressive sensing to the present context.

\begin{lemma}\label{Concentration} For any $c\in \RR^N$ and any real number $0<\eta <1$ the following inequality holds,

\[\PP\left\{ \left| \|\Phi(\omega) c\|_2^2 - c^tVc \right|\geq \eta c^tVc\right\}\leq  2e^{-m c_0(\eta)}\]
where $c_0(\eta):=\min\left(\frac{1}{2}\log\left(\frac{1}{1+\eta}\right)+\frac{\eta}{2}, \frac{1}{2}\log\left(\frac{1}{1-\eta}\right)-\frac{\eta}{2}\right).$
\end{lemma}
\begin{proof} The components of $\Phi(\omega)c$ are independent normal random variables with mean zero and common variance $\frac{c^tVc}{m}$. It follows that $Z:=\frac{m\|\Phi(\omega)c\|^2}{c^tVc}$ has a Chi-squared distribution with $m$ degrees of freedom. It follows that for any $t<\frac{1}{2}$ the equality $\EE[e^{tZ}]=\frac{1}{(1-2t)^{\frac{m}{2}}}$ holds. As a result, for every real number $\alpha$ and $t>0$  we have
\[\PP\{ Z>\alpha \}=\PP\{e^{tZ}>e^{t\alpha}\}\leq \EE[e^{tZ}]e^{-t\alpha }=\frac{e^{-t\alpha}}{(1-2t)^\frac{m}{2}} \]
If $\alpha=m(1+\eta)$ then the right hand side equals $e^{m\phi(t)}$ where 
\[\phi(t)=-\frac{1}{2}\log(1-2t) -t(1+\eta)\]
The function $\phi(t)$ is strictly convex and $\phi'(t^*)=0$ when $t^*=\frac{\eta}{2(1+\eta)}$. Setting $t=t^*$ in the above formula we obtain an upper bound of
$\exp\left(-m\left(\frac{1}{2}\log(\frac{1}{1+\eta})+\frac{\eta}{2}\right)\right)$. It is shown similarly that for $\eta>0$
\[\PP\left\{Z\leq m(1-\eta)\right\}\leq \exp\left(-m\left(\frac{1}{2}\log\left(\frac{1}{1-\eta}\right)-\frac{\eta}{2}\right)\right)\]
The claimed inequality now follows immediately from the union bound and the definition of $Z$.
  
\end{proof}

We are now ready to prove the main result of this section as stated in the introduction.

\begin{theorem*}[{\bf B}] For any real number $0<\delta<\frac{1}{2} $, any integer $k$, and all sufficiently large $d$ the following inequality holds
\[ \PP\left\{\delta_{2k}(\Phi)>\delta\right\}< \binom{N}{2k}\left(\frac{30}{\delta}\right)^{2k}2e^{-mc_0\left(\frac{\delta}{6}\right)}\] 
In particular, for all sufficiently large $d$, there exist matrices $\Phi$ such that $\delta_{2k}(\Phi)<\delta$ whenever 
\[m\geq \frac{2k\log\left(\frac{30eN}{2k\delta}\right)+\log(2)}{c_0\left(\frac{\delta}{6}\right)}\]
\end{theorem*}
\begin{proof} Fix a set $T\subseteq [N]$ with $|T|=2k$. From the theory of covering numbers it is well-known that there exists a set of points $Y\subseteq S^{N-1}\subseteq \RR^N$ such that:
\begin{enumerate}
\item The points of $Y$ are supported on $T$.
\item For every $z\in S^{N-1}$ with support on $T$ we have $\inf_{y\in Y}\|y-z\|_2<\frac{\delta}{5}$
\item $|Y|\leq (30/\delta)^{2k}$.
\end{enumerate}
By Lemma~\ref{Concentration} and a union bound the probability of the $\omega\in \Omega$ such that  
\[|\|\Phi(\omega)y\|_2^2-y^tVy|>\frac{\delta}{6}y^tVy\] for some $y\in Y$ is bounded above by $\left(\frac{30}{\delta}\right)^{2k}2e^{-mc_0\left(\frac{\delta}{6}\right)}$.
Moreover, by Lemma~\ref{Gershgorin} there exists an integer $d_0$ such that for $d>d_0$ the following two inequalities hold. 
\[
\begin{array}{l}
\left(1+\frac{\delta}{6}\right)\lambda_{\max}(V)\leq 1+\frac{\delta}{5}\\ 
\left(1-\frac{\delta}{6}\right)\lambda_{\min}(V)\geq 1-\frac{\delta}{5}
\end{array}
\]

We conclude that for all such $d$ the probability of the event $E_T$, consisting of the $\omega\in \Omega$ such that 
\[|\|\Phi(\omega)y\|_2^2-\|y\|^2|>\frac{\delta}{5}\|y\|^2\] for some $y\in Y$ is bounded above by $\left(\frac{30}{\delta}\right)^{2k}2e^{-mc_0\left(\frac{\delta}{6}\right)}$
We will show that if $\omega\not\in E_T$ then the inequality 
\[(1-\delta)\|c\|^2<\|\Phi(\omega)c\|_2^2<(1+\delta)\|c\|^2\]
holds for every $c$ supported on $T$.
To this end, let $A$ be the smallest real number such that for every $c$ supported on $T$, the inequality $\|\Phi(\omega)c\|_2\leq \sqrt{1+A}\|c\|_2$
holds. We will show that $A < \delta$ by estimating $\|\Phi(\omega)c\|_2$ for $c\in S^{N-1}\subseteq\RR^N$ with support on $T$. If $y^*\in Y$ is such that $\|c-y\|_2\leq \delta/5$ then the following inequalities hold 
\[ \|\Phi(\omega)c\|_2\leq \|\Phi(\omega)y^*\|_2+ \|\Phi(\omega)(c-y^*)\|_2\leq \sqrt{1+\delta/5}+\sqrt{1+A}(\delta/5) \]
From the definition of $A$ it follows that the inequality 
\[\sqrt{1+A}\leq \sqrt{1+\delta/5}+\sqrt{1+A}(\delta/5) \]
holds and thus
\[A\leq \frac{(1+\frac{\delta}{5})}{(1-\frac{\delta}{5})^2}-1\leq \frac{\frac{3\delta}{5}}{1-2\frac{\delta}{5}}<\delta\]
so that $\|\Phi(\omega)c\|_2\leq \sqrt{1+\delta}\|c\|$ for every $c$ supported on $T$.
For the opposite inequality we have
\[\|\Phi(\omega)c\|_2\geq \|\Phi(\omega)y^*\|_2- \|\Phi(\omega)(c-y^*)\|_2\geq \sqrt{1-\delta/5}-\sqrt{1+\delta}(\delta/5)=:b\]
and the last quantity $b$ is bounded below by $\sqrt{1-\delta}$ because
\[1-b^2=\frac{\delta}{5}-(1+\delta)(\delta/5)+2(\delta/5)\sqrt{(1-\frac{\delta}{5})(1+\delta)}\leq \frac{\delta(2+\delta)}{5}\leq 3\frac{\delta}{5}<\delta.\]
We conclude that the $\omega\in \Omega$ for which $\delta_{2k}(\Phi)>\delta$ is contained in the union of the $E_T$ as $T$ ranges over the $\binom{N}{2k}$ subsets of $[N]$ of size $2k$ and the Theorem follows from the union bound.  
For the last part recall that $\binom{N}{2k}\leq \left(\frac{Ne}{2k}\right)^{2k}$.

\end{proof}

\begin{remark} The value of the required degree $d$ can be easily estimated from the explicit bound in Lemma~\ref{Gershgorin}.
\end{remark}

\begin{corollary}\label{ExactRecovery} Let $M:=\Phi$ in the recovery algorithms from Theorem~\hyperref[CompressedComputing]{A}. For all sufficiently large $d$, the failure probability is bounded above by 

\[\binom{N}{2k}\left(\frac{30}{\sqrt{2}-1}\right)^{2k}2e^{-mc_0\left(\frac{\sqrt{2}-1}{6}\right)}.\]

\end{corollary}
\begin{proof} Follows immediately from Theorem~\hyperref[CompressedComputing]{A} and the inequality in Theorem~\hyperref[main]{B} with $\delta=\sqrt{2}-1$.
\end{proof}

\begin{remark} While having small restricted isometry constants as above is a sufficient condition for a matrix to be suitable for compressive sensing this condition is by no means necessary. In particular, the restricted isometry property is unable to explain the exact shape of the well-known phase transition phenomena that occur in compressive sensing problems (i.e. the existence of a hard threshold on the number of measurements above which the convex recovery procedure is generally successful and below which the convex recovery procedure is generally unsuccessful). A much more satisfactory approach to these questions is given by classical integral geometry (see for instance~\cite{ALMT}). It would be very interesting to use these methods to better understand phase transitions for the matrices $\Phi$ suggested by our numerical experiments in Figure~\ref{Fig: exactRecovery}.
\end{remark}

The approximate recovery algorithm from the previous corollary can be used to find good approximations of  a measure $\nu$ supported on $k$ of the points $q_i$ which best approximates $\mu$ in the sense that $\|b_\mu-b_\nu\|_2$ is as small as possible. Note, however that the vectors $b_{\bullet}$ depend not only on the measure but also on the sequence of functions we use for computing them. When using the random measurement matrix $\Phi$ it may be thus difficult to interpret the quantity $\|b_\mu-b_\nu\|_2$ which in this setting becomes a random variable. In the following section we give a geometric interpretation for {\it the mean of this random variable} and show that its values concentrate around it allowing us to clarify the outcome of the approximate recovery algorithm.

\subsection{Optimal mean-square error approximations.}
If the sequence of functions $f_i$ used for moment computations is a sequence of random functions $f_i(\omega)$ then we can define the following concept of ``closeness" between point measures.

\begin{definition} Let $f_i(\omega)$, $1\leq i \leq m$ be a random family of real valued functions on the sphere $S$. If $\nu$ and $\mu$ are point measures then we define the mean-squared error between $\nu$ and $\mu$ to be $\EE[\|b_\mu-b_\nu\|_2^2]$.
\end{definition}

The following Lemma shows that when the $f_i(\omega)$ are an independent sample of Kostlan-Shub-Smale polynomials $f_i:=\frac{P(x)}{\sqrt{m}}$  of size $m$ then there is a closed expression for the mean-squared error. Remarkably these expressions depend only on the locations of the points and the degree $d$ of our random polynomials.

More precisely let $\mu=\sum_{i=1}^k r_i \delta_{y_i}$ and suppose $\nu =\sum_{i=1}^N c_i \delta_{q_i}$. Let $V\in \RR^{N^2}$, $A\in \RR^{N\times k}$ and $D\in \RR^{k^2}$ be matrices with entries given by $V_{st}=\left(\frac{1+\langle q_s, q_t\rangle}{2}\right) ^d$, $A_{st}=\left(\frac{1+\langle y_s, q_t\rangle}{2}\right)$ and $D_{s,t} = \left(\frac{1+\langle y_s, y_t\rangle}{2}\right)$ respectively.

\begin{lemma} The mean-squared error $\EE[\|b_\mu-b_\nu\|_2^2]$ is given by the quadratic form
\[\Psi(r,c)=c^tVc -2c^tAr+r^tDr.\]
Moreover, if $S\subseteq [N]$ is a set of size $k$ such that the matrix $V_S$ is invertible then there is a unique signed measure $\nu^*$ supported on $\{q_i:i\in S\}$ for which the mean-squared error is minimized. It's non-zero coefficients are given by $c_S^*:=V_S^{-1}A^Sr$.
\end{lemma} 
\begin{proof} Let $H$ be the $m\times (N+k)$ matrix with columns indexed by $q_1,\dots, q_N, y_1,\dots, y_k$ with entries given by
\[H_{ij}= \begin{cases}
P_i(q_j)\text{, if $j\leq N$}\\
P_i(y_{j-N})\text{, if $j>N$.}
\end{cases}\]
Arguing as in the proof of Lemma~\ref{Basic} we conclude that the following equalities hold
\[\Psi(r,c)=\EE\left[ \left\|H \left(
\begin{array}{c}
c\\
-r\\
\end{array}
\right)\right\|^2\right]=c^tVc -2c^tAr+r^tDr\]
proving the first claim.
If $r$ is fixed and $c$ is a vector supported on the set $S$ then the quadratic form becomes
\[\psi(c)=c^tV_Sc -2c^tA^Sr+r^tDr\]
where $A^S$ is the restriction of $A$ to the rows corresponding to points $q_i$, $i\in S$.
If $V_S$ is invertible this function is strictly convex and thus its unique minimum is achieved when $\nabla \psi(c^*)=0$, proving the second claim. \end{proof}

If $V_S$ is invertible for all $S$ of size $k$ we can therefore define an optimal mean-squared error approximation to $\mu$ supported on $k$ points.

\begin{definition} An optimal mean-squared error approximation to $\mu$ supported on $k$ of the points $q_i$ is a signed measure $\nu^*$ supported on the $q_i$ whose vector of coefficients $c_{\nu^*}$ satisfies:
\begin{enumerate}
\item The support of $c_{\nu^*}$ has size at most $k$
\item $(c_{\nu^*})_S=c_S^*$ for some $S\subseteq [N]$ of size $k$.
\item $\Psi(c_{\nu^*},r)=\min_{S\subseteq [N], |S|=k} \Psi(c_S^*,r)=:(\tau^*)^2$
\end{enumerate}
We call the number $\tau^*$ the smallest mean squared approximation error. Note that this quantity depends only on the set of points $C$, the integer $d$ and the measure $\mu$.
\end{definition}
The following Theorem relates the optimal solution of our approximate recovery algorithm with the optimal mean-squared error approximation,
\begin{theorem} Let $\epsilon>0$ be a real number. Let $c^*$ be a minimizer of the problem
\[\min \|c\|_1\text{ subject to $\|\Phi c-b_\mu\|_2\leq (1+
\epsilon) \tau^*$.}\]
then $\|c^*-c_{\nu^*}\|_2\leq (1+\epsilon)B_1\tau^*$ for all sufficiently large $d$ with probability at least
\[ 1- \binom{N}{2k}\left(\frac{30}{\sqrt{2}-1}\right)^{2k}2e^{-mc_0\left(\frac{\sqrt{2}-1}{6}\right)} - 2e^{-mc_0(\epsilon)}\]
where $c_0(\eta):=\min\left(\frac{1}{2}\log\left(\frac{1}{1+\eta}\right)+\frac{\eta}{2}, \frac{1}{2}\log\left(\frac{1}{1-\eta}\right)-\frac{\eta}{2}\right).$
\end{theorem}
\begin{proof} If $\delta_{2k}(\Phi)<\sqrt{2}-1$ and $b_{\nu^*}:=\Phi c_{\nu^*}$ satisfies $\|b_{c_{\nu^*}}-b_\mu\|\leq (1+\epsilon)\tau^*$ then the conclusion follows from Theorem~\hyperref[CompressedComputing]{A} part $(2)$. The probability that either of those fails is at most the sum of the probabilities computed in Corollary~\ref{ExactRecovery} and Lemma~\ref{Concentration} proving the Theorem.
\end{proof}

\section{Approximate recovery of probability measures.}
\label{ApproxRecov}
The value of the optimal mean-squared error approximation $\tau^*$ from the previous section seems difficult to compute and to interpret. In this section we show that, under the additional assumption that the points of our code $C=\{q_1,\dots, q_N\}$ are sufficiently close to those in the support of a probability measure $\mu:=\sum_{i=1}^k g_i \delta_{x_i}$, then our algorithm recovers an approximation of the measure $\mu_{C}$ which is supported on $C$ and which is closest to $\mu$ in the Wasserstein distance. 

More precisely we let $\mu_C$ be the probability measure that has the same coefficients as $\mu$ placed at the points of $C$ closest (in the usual metric $d(x,y)$ on the sphere) to the support of $\mu$. Relabeling the points of $C$ if necessary $\mu_C:=\sum_{i=1}^k g_i \delta_{q_i}$ where $q_i$ is any point of $C$ closest to $x_i$. Recall that the Wasserstein distance between probability measures $\nu_1,\nu_2$ supported on the sphere is given by 
\[W(\nu_1,\nu_2):=\sqrt{\inf_{\lambda}\int_{S\times S} d(X,Y)^2d\lambda}\]
where $(X,Y)$ is any vector with probability distribution $\lambda$ such that $X$ and $Y$ have marginal distributions given by $\mu$ and $\nu$ respectively. It is well known (see for instance~\cite[page 33]{FQPD}) that $W(\mu,\mu_C)\leq W(\mu,\tau)$ for any other probability measure $\tau$ supported on $C$. It is a problem of much interest to be able to find such optimal approximations $\mu_C$.

In this section we show that our approximate recovery algorithm can be used for finding an approximation of $\mu_C$ via convex programming whenever the support of $\mu$ is sufficiently close to $C$. For a positive integer $m$ let $P_1,\dots, P_m$ be independent Kostlan-Shub-Smale polynomials. For a measure $\mu$ let $b_{\mu}$ be the (random) vector of moments of $\mu$ with respect to the functions $\frac{1}{\sqrt{m}} P_i$. Our main result is Theorem~\hyperref[main:Approx]{C} which gives an estimate for the norm $\|c^*-g\|_2$ which holds with overwhelming probability as the number of measurements increases. 

The key result is the following Lemma which estimates the average mean squared error between the random vectors $b_{\mu}$ and $b_{\mu_C}$ in terms of geometric quantities. To this end let $\theta=\max_j\min_i\arccos \langle x_j,q_i\rangle$ (due to our notational conventions $\theta=\max_i \arccos \langle x_i,q_i\rangle$).

\begin{lemma}\label{lemma:MSE}  The following inequality holds
\[\EE[\|b_{\mu}-b_{\mu_C}\|_2^2]\leq \|g\|_2^2 2k\left(1-\left(\frac{1+\cos(\theta)}{2}\right)^d\right).\] 
In particular if $\mu$ is a probability measure then
\[\EE[\|b_{\mu}-b_{\mu_C}\|_2^2]\leq 2k\left(1-\left(\frac{1+\cos(\theta)}{2}\right)^d\right).\] 
and if $\mu$ is a uniform probability measure then
\[\EE[\|b_{\mu}-b_{\mu_C}\|_2^2]\leq 2\left(1-\left(\frac{1+\cos(\theta)}{2}\right)^d\right).\] 
\end{lemma}
\begin{proof} By definition we know that

\[\EE[\|b_{\mu}-b_{\mu_C}\|_2^2]=\EE\left(\sum_{j=1}^m\frac{1}{m} \left(\sum_{i=1}^kg_i(P_j(x_i)-P_j(q_i))\right)^2\right)\]
and using the Cauchy-Schwarz inequality we conclude that the following inequality holds
\[\left(\sum_{i=1}^k g_i(P_j(x_i)-P_j(q_i))\right)^2\leq \|g\|_2^2 \|\left(P_j(x_i)-P_j(q_i)\right)_{1\leq i\leq k}\|_2^2 \]
for $j=1,\dots, m$. As a result
\[\EE\left(\sum_{i=1}^k g_i(P_j(x_i)-P_j(q_i))\right)^2\leq \|g\|_2^2 \sum_{i=1}^k\EE\left[(P_j(x_i)-P_j(q_i))^2\right]\] 
By Lemma~\ref{Basic} the expected value in the right hand side can be estimated, for all $i$ and $j$ as
\[\EE\left[(P_j(x_i)-P_j(q_i))^2\right]=2-2\left(\frac{1+\langle x_i,q_i\rangle}{2}\right)^d\leq 2\left(1-\left(\frac{1+\cos(\theta)}{2}\right)^d\right)\]
proving the first claim. For the last two claims recall that $\|c\|_2\leq \|c\|_1$ and that for a uniform probability measure $\|c\|_2=\frac{1}{\sqrt{k}}$
\end{proof}

Combining our previous results we will now prove the main result of this section,

\begin{theorem*}[{\bf C}] 

Let $g\in \RR^n$ be the vector of coefficients of the measure $\mu_C$ and let $\epsilon>0$. 
If $c^*$ is a minimizer of the problem
\[\min \|c\|_1\text{ subject to $\|\Phi c-b_\mu\|_2\leq (1+
\epsilon)  \|g\|_2 \sqrt{2k\left(1-\left(\frac{1+\cos(\theta)}{2}\right)^d\right)}$.}\]
then
\[ \|g-c^*\|_2\leq B_1(1+\epsilon)\|g\|_2\sqrt{2k\left(1-\left(\frac{1+\cos(\theta)}{2}\right)^d\right)}\]
for all sufficiently large $d$ with probability at least
\[ 1- \binom{N}{2k}\left(\frac{30}{\sqrt{2}-1}\right)^{2k}2e^{-mc_0\left(\frac{\sqrt{2}-1}{6}\right)} - 2e^{-mc_0(\epsilon)}\]
where $c_0(\eta):=\min\left(\frac{1}{2}\log\left(\frac{1}{1+\eta}\right)+\frac{\eta}{2}, \frac{1}{2}\log\left(\frac{1}{1-\eta}\right)-\frac{\eta}{2}\right).$
\end{theorem*}
\begin{proof} If $\delta_{2k}(\Phi)<\sqrt{2}-1$ and $\|b_{\mu_{C}}-b_\mu\|\leq (1+\epsilon)  \|g\|_2\sqrt{2k\left(1-\left(\frac{1+\cos(\theta)}{2}\right)^d\right)}$ then the conclusion follows from Theorem~\hyperref[CompressedComputing]{A} part $(2)$. The probability that either of those fails is at most the sum of the probabilities computed in Corollary~\ref{ExactRecovery} and Lemma~\ref{Concentration} proving the Theorem.
\end{proof} 

Finally, we show that successive application of our approximation algorithm using an increasing sequence $C_{j}$ of codes with dense union leads to a consistent estimation of certain classes of probability measures $\mu$. More precisely we use our algorithm to construct a sequence of measures $\mu_{C_{j}}^*$ supported on $C_{j}$ which converges to $\mu$ in the Wasserstein metric whenever $\mu$ is easily approximable (in a sense to be defined) by the sequence of codes $\left(C_{j}\right)_j$. 

Recall that $\mu:=\sum_{i=1}^k g_i \delta_{x_i}$ and assume $(C_j)_{j\in\mathbb{N}}$ is an increasing sequence of finite subsets of the sphere whose union $\bigcup C_{j}$ is dense in $S$.  Suppose that $C_{j}$ consistis of $N_j$ points labeled $\{q_1,\dots, q_{N_j}\}$. Define $\alpha_j:=\max_{x\neq x'\in C_j} \langle x,x'\rangle$ and $\theta_j:=\max_{i=1,\dots,k}\min_{q\in C_{j}} \arccos\langle q,x_i\rangle$. Note that $\alpha_j<1$ and that by density $\lim_{j\rightarrow \infty}\alpha_j = 1$. 

\begin{definition} 
\label{EasyApprox}
We say that the measure $\mu$ is easily approximable by the sequence of codes $\left(C_{j}\right)_{j\in \mathbb{N}}$ if there exist a sequence of integers $d_j$ such that:

\begin{enumerate}
\item For all sufficiently large $j$ the inequality $(k-1)\left(\frac{1+\alpha_j}{2}\right)^{d_j}<\sqrt{2}-1$ holds.
\item The equality $\lim_{j\rightarrow \infty} \left(\frac{1+\cos(\theta_j)}{2}\right)^{d_j}=1$ holds. 
\end{enumerate}
\end{definition}

\begin{remark} The numbers $\alpha_j$ and $\cos(\theta_j)$ measure the distance between distinct points in the code $C_j$ and the distance between the points of the support of $\mu$ and the code $C_j$. Intuitively, a measure $\mu$ is easily approximable by the codes $C_j$ if, as $j\rightarrow \infty$ the points of the code approach the support of $\mu$ faster than they approach each other.
\end{remark}

For each $j$ let $M^{(j)}$ be the measurement matrix defined by an independent sample of size $m_j$ of Kostlan-Shub-Smale polynomials of degree $d_j$ and let $b_{\mu}^{(j)}$ be the corresponding vector of moments. For a real number $\tau_j$ define $c^*_j$ to be an optimum of the optimization problem
\[\min \|c\|_1 \text{ s.t. $\|M^{(j)}c-b_{\mu}^{(j)}\|\leq \tau_j$}.\]

We do not know a way to guarantee that the optima $c_j^*$ in the previous problem are vectors with nonnegative entries (although our computational experiments suggest that this is generally the case, up to numerical noise). We therefore construct a probability measure out of the vector $c_j^*$ and a numerical threshold parameter $t>0$ as follows: let $(h_j^*)_i:=(c_j^*)_i$ if $(c_j^*)_i>t$ and $(h_j^*)_i:=0$ otherwise and define $\mu_{C_j}^*:=\frac{\sum_{i=1}^N (h_j^*)_i\delta_{q_i}}{\|h_j^*\|_1}$. 

\begin{theorem}\label{thm: consistency}  If $\mu$ is easily approximable by $(C_j)_{j\in \mathbb{N}}$ and $t:=\frac{1}{2}\min_{i=1,\dots,k} (g_i)$ then there exist sequences $m_j$ and $\tau_j$ such that the sequence of probability measures $\mu_{C_j}^*$ converges to $\mu$ in the Wasserstein metric almost surely.
\end{theorem}
\begin{proof} Since $\mu$ is easily approximable by the $C_j$ there exists a sequence $d_j$ which satisfies the two items in Definition~\ref{EasyApprox}.
For $\epsilon>0$ choose a sequence of integers $m_j$ which are sufficiently large so that the probability of failure in the inequality of Theorem $C$ is bounded above by the quantity

\[ \binom{N_j}{2k}\left(\frac{30}{\sqrt{2}-1}\right)^{2k}2e^{-m_jc_0\left(\frac{\sqrt{2}-1}{6}\right)}+2e^{-m_jc_0(\epsilon)}\]
and so that the sum over all integers $j$ of this quantity converges.
Define $\tau_j$ by the formula
\[\tau_j:=(1+
\epsilon)  \|g\|_2 \sqrt{2k\left(1-\left(\frac{1+\cos(\theta_j)}{2}\right)^{d_j}\right)}\]
and note that $\lim_{j\rightarrow \infty}\tau_j = 0$ since we are assuming that $\mu$ is a measure which is easily approximable by $(C_j)_j$.

By the Borel-Cantelli Lemma we conclude that with probability one the coefficients of the resulting sequence of optima $c_j^*$ satisfy $\|c_j^*-g\|_2\leq B_1\tau_j$ for all sufficiently large $j$. It follows that there exists an integer $j_0$ such that for all $j\geq j_0$ the inequality $\|c_j^*-g\|_2<t$ holds. For all such $j$ the supports of $h_j^*$ has cardinality $k$ because $\|g-c_j^*\|_{\infty}\leq \|g-c_j^*\|_2<t$ so the cutoff procedure keeps exactly those coefficients $c_i$ for which $g_i\neq 0$. Next, note that the following inequalities hold for $j\geq j_0$,
\[|\|h_j^*\|_1-\|g\|_1|\leq \|g-h_j^*\|_1\leq \sqrt{k}\|g-h_j^*\|_2\leq\sqrt{k} \|g-c_j^*\|\leq \sqrt{k}B_1\tau_j\]
and as a result $\lim_{j\rightarrow\infty} \|h_j^*\|_1=1$ and $\lim_{j\rightarrow \infty} \left\| h_j^*-\frac{h_j^*}{\|h_j^*\|_1}\right\| =0$. 
Moreover the inequality 
\[\left\| g-\frac{h_j^*}{\|h_j^*\|_1}\right\|_1\leq \|g-h_j^*\|_1+\left\| h_j^*-\frac{h_j^*}{\|h_j^*\|_1}\right\|_1 \]
holds for all $j\geq j_0$ so $\lim_{j\rightarrow \infty}\left\| g-\frac{h_j^*}{\|h_j^*\|_1}\right\|_1=0$.

Now let $\mu_{C_j}$ be the best approximation to $\mu$ in the Wasserstein distance among measures supported in $C_j$ (i.e. $\mu_{C_j}:=\sum_{i=1}^k g_i\delta_{q_i^j}$ where $q_i^j$ is a point of $C_j$ closest to $x_i$) and note that the inequality 

\[W(\mu_{C_j},\mu_{C_j}^*)\leq \pi \left\|g-\frac{h_j^*}{\|h_j^*\|_1}\right\|_1\]
holds because masses supported on different points on the sphere need to be transported at most the diameter $\pi$ of the unit sphere. We conclude that $W(\mu_{C_j},\mu_{C_j}^*)\rightarrow 0$ as $j\rightarrow \infty$. Since the Wasserstein metric satisfies the triangle inequality we have
\[W(\mu,\mu_{C_j}^*)\leq W(\mu,\mu_{C_j})+W(\mu_{C_j},\mu_{C_j}^*).\]
The first term in the right-hand side goes to zero by density of the set $\bigcup_j C_j$ in the sphere and we have proven that the second term goes to zero verifying the claim.
\end{proof}

\section{Some numerical experiments}
\label{Sec: numericAlgos}

To illustrate our main Theorems and to explore the numerical behavior of the proposed algorithms we carried out some computer experiments on point measures on the unit spheres in $\RR^2$ and $\RR^8$. We fix codes $C$ in $S^1$ and $S^7$. For the circle we let $C$ be a set of $200$ equally-spaced points and for $S^7$ we let $C$ be the $240$ vectors of the root system $E_8$ (see~\cite[Chapter 3]{Humphreys} for an introduction to root systems and~\cite{Wiki} for a purely combinatorial description of this remarkable code). We carry out the following two numerical experiments:

\begin{enumerate}
\item Our first experiment illustrates the exact recovery algorithm. We consider point measures $\mu=\sum c_i\delta_{q_i}$ supported on $k$ points of $C$ with $k=1,\dots, 50$. We try to recover $\mu$ via the exact recovery procedure proposed in Theorem~\hyperref[CompressedComputing]{A} using $m$ moments with respect to a random sample of Kostlan-Shub-Smale polynomials of degrees $d=30,5$ respectively. We denote our optimal solution by $c^*$ and report the error $\|c-c^*\|_2$.  Figure~\ref{Fig: exactRecovery} contains the results of these experiments for point measures in the one-dimensional and seven-dimensional spheres respectively. As expected from our estimates of the RIP constants (see Theorem~\hyperref[main]{B}) the recovery is exact for several sparsities and the range increases considerably as the number $m$ of measurements increases.

\item Our second experiment illustrates the approximate recovery algorithm. We consider point measures $\mu=\sum_{i=1}^k c_i\delta_{p_i}$ supported on $k$ points $p_i$ which do not lie in $C$.  We use the approximate recovery procedure proposed in Theorem~\hyperref[CompressedComputing]{A} given $m$ moments of $\mu$ with respect to a random sample of Kostlan-Shub-Smale polynomials of degrees $d=30,5$ respectively. To do this, we solve the optimization problem
\[\min \|c\|_1\text{ subject to $\|Mc-b_\mu\|_2\leq 
\tau$.}\]
for a small value of $\tau$. In order to determine the value of $\tau$ we recommend the following procedure:
\begin{enumerate}
\item Choose a small value of $\epsilon$ ($\epsilon=0.1$ in our experiments) and solve the optimization problem above for several $\tau\in [0,\epsilon)$. For each optimal solution $c^*(\tau)$ let $k(\tau)$ be the ``numerical" sparsity (i.e. the number of coefficients with absolute value above a certain numerical error threshold). Experiments show that the numerical sparsity tends to stabilize around some value $k^*$ for small $\tau$. Figure~\ref{Fig: approxRecoverySupport}  shows the function $k(\tau)$ for measures supported in $k=3,30$ points in $S^1$ and $S^7$ respectively when the support of $\mu$ lies at a distance (angle) of $\frac{2\pi}{400\times 20}$ and $\frac{2\pi}{ 240\times 2}$ from the code $C$ and the number $m$ of measurements is $m=120,144$ respectively.
\item Solve the optimization problem with $\tau=\tau^*$ where $\tau^*$ is the smallest value  of $\tau$ for which the solution achieves the stabilized sparsity $k^*$.
\end{enumerate}
In Figure~\ref{Fig: approxRecoveryError} we compare the measure $\mu_C=\sum c_i\delta_{q_i}$ defined as the best approximation to $\mu$ supported on $C$ (as in Section~\ref{ApproxRecov}) and the measure $\mu^*=\sum c_i^*\delta_{q_i}$ obtained from solving the convex problem above with $\tau=\tau^*$. The error reported is $\|c-c^*\|_2$ and we let $\mu$ range over a set of measures supported on sets of size $k=3,30$ resp. whose minimum distance with $C$ is at least a given value $\theta$ which we vary between $0\leq \theta\leq \frac{\pi}{200}$ and $0\leq \theta\leq \frac{\pi}{240}$ respectively. As expected from Theorem~\hyperref[main:Approx]{C}, the accuracy of the recovery improves when the support of the unknown measure is closer to the points of $C$. 
\end{enumerate}

All algorithms were implemented in the Julia programming language~\cite{Julia} using the JuMP~\cite{JuMP} modeling language. The resulting second-order cone programs were solved with the Mosek large scale optimization solver on a personal computer. The Julia code for the computational experiments of this section is available for download at \url{https://github.com/hernan1992garcia/measure-recovery}.

\begin{figure}
\hspace*{-1.75cm}
\caption{Error in exact recovery algorithm in the spheres $S^1$ and $S^7$}
\label{Fig: exactRecovery}
\includegraphics[height=6cm]{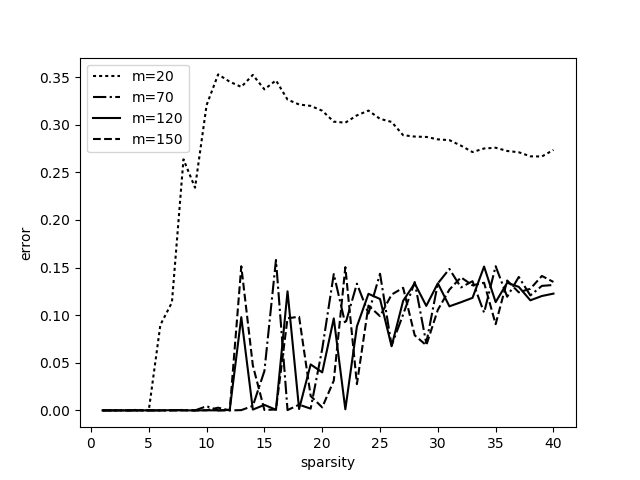}\includegraphics[height=6cm]{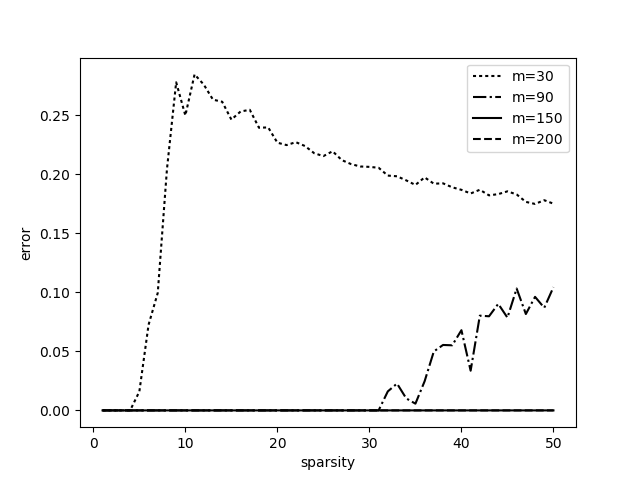}

\end{figure}

\begin{figure}
\hspace*{-1.75cm}
\caption{Sparsity of optimal solutions of the approximate recovery algorithm in spheres $S^1$ and $S^7$}
\label{Fig: approxRecoverySupport}
\includegraphics[height=6cm]{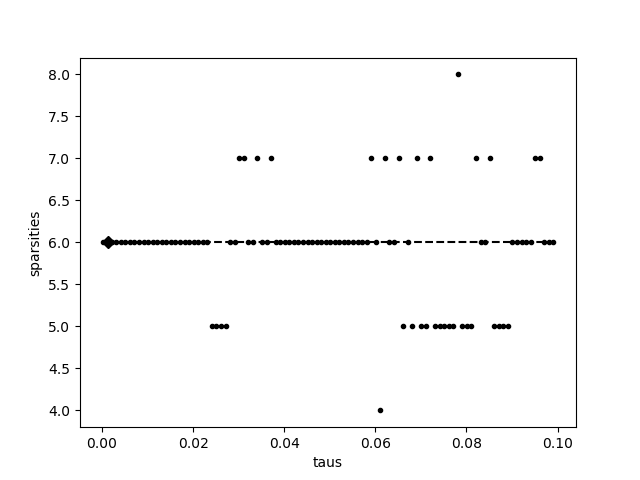}\includegraphics[height=6cm]{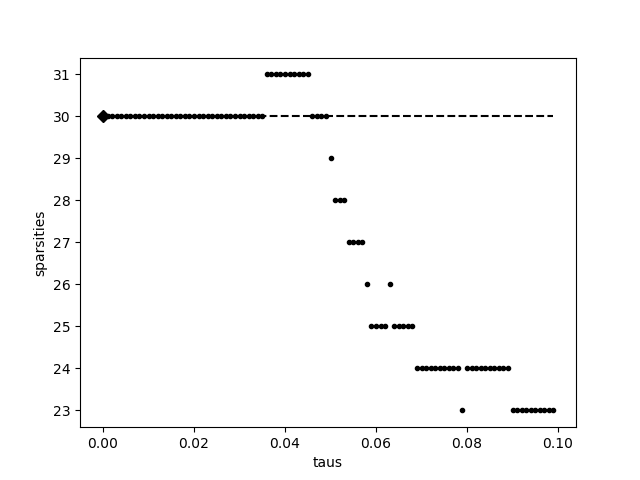}
\end{figure}

\begin{figure}
\hspace*{-1.75cm}
\caption{Error in approximate recovery algorithm for measures in the spheres $S^1$ and $S^7$}
\label{Fig: approxRecoveryError}
\includegraphics[height=6cm]{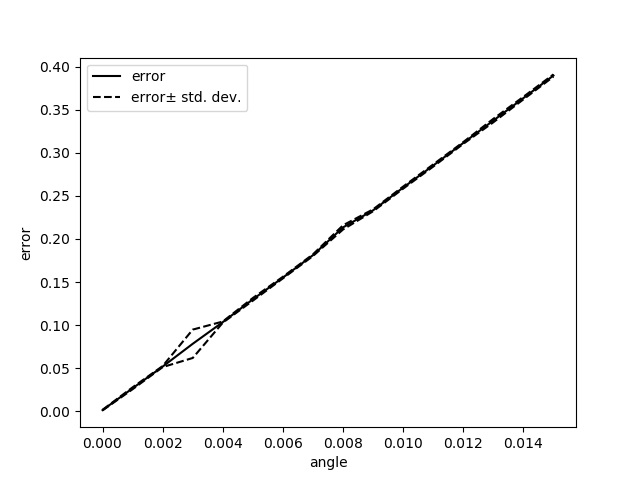}\includegraphics[height=6cm]{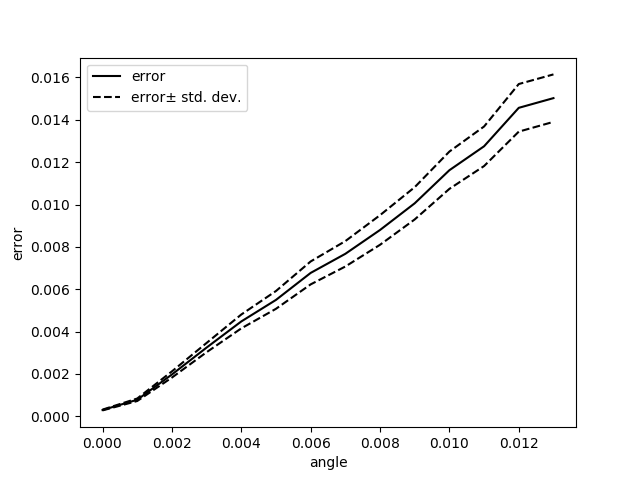}
\end{figure}

\raggedright

\end{document}